\documentclass[11pt]{amsart}
\usepackage{amssymb,latexsym,amsmath}
\usepackage[mathscr]{eucal}
\usepackage{xcolor}
\newtheorem{theorem}{Theorem}

\newtheorem{proposition}[theorem]{Proposition}
\newtheorem{corollary}[theorem]{Corollary}
\newtheorem*{claim}{Claim}
\newtheorem{lemma}[theorem]{Lemma}

\numberwithin{theorem}{section}
\numberwithin{equation}{section}

\DeclareFontFamily{OT1}{rsfs}{} \DeclareFontShape{OT1}{rsfs}{m}{n}{
<-7> rsfs5 <7-10> rsfs7 <10-> rsfs10}{}
\DeclareMathAlphabet{\mycal}{OT1}{rsfs}{m}{n}

\begin{document}

\title{A sharp inscribed radius estimate for fully nonlinear flows}
\author{Simon Brendle and Pei-Ken Hung}
\address{Department of Mathematics \\ Columbia University \\ New York, NY 10027}
\address{Department of Mathematics \\ Columbia University \\ New York, NY 10027}
\thanks{The first author was supported in part by the National Science Foundation under grant DMS-1505724.}

\begin{abstract}
We prove a sharp estimate for the inscribed radius under certain fully nonlinear curvature flows. This estimate is asymptotically sharp on cylinders.
\end{abstract}
\maketitle

\section{Introduction}

Given a hypersurface in Euclidean space and a point $p$ on that hypersurface, the inscribed radius at $p$ is defined to be the radius of the largest ball that lies inside the hypersurface and touches the hypersurface at $p$. It follows from deep results of Brian White \cite{White1},\cite{White2} that, for an embedded mean convex solution of mean curvature flow, the inscribed radius is bounded from below by $\frac{\alpha}{H}$, where $\alpha>0$ is a uniform constant that depends only on the initial data. An alternative proof of that fact was given by Sheng and Wang \cite{Sheng-Wang}. In that paper, Sheng and Wang also introduced the notions of noncollapsing and inscribed balls. Later, Andrews \cite{Andrews} gave another proof of the inscribed radius estimate, which relies on a two-point maximum principle. This technique was pioneered in Huisken's work on the curve shortening flow \cite{Huisken} (see also \cite{Hamilton2}); a survey can be found in \cite{Brendle-survey}. An interesting feature of the argument in \cite{Andrews} is that it can be extended to certain fully nonlinear curvature flows; this was done in \cite{Andrews-Langford-McCoy}. In \cite{Brendle1}, the first author obtained a sharp estimate for the inscribed radius for embedded mean convex solutions of mean curvature flow. More precisely, given any $\delta>0$, it turns out that the inscribed radius is bounded from below by $\frac{1}{(1+\delta)H}$ at each point where the curvature is sufficiently large. An alternative proof was subsequently given by Haslhofer and Kleiner \cite{Haslhofer-Kleiner}.

In this paper, we extend the estimate in \cite{Brendle1} to certain fully nonlinear flows. Throughout this paper, we fix a constant $\kappa \geq 0$. We consider a family hypersurfaces in a Riemannian manifold $X$ which are $\kappa$-two-convex in the sense that $\lambda_1+\lambda_2 > 2\kappa$ and which move with velocity 
\[G_\kappa= \Big ( \sum_{i < j}\frac{1}{\lambda_i+\lambda_j-2\kappa} \Big )^{-1},\] 
where $\lambda_1 \leq \hdots \leq \lambda_n$ denote the principal curvatures. This flow was introduced in \cite{Brendle-Huisken}. Unlike mean curvature flow, this flow preserves $\kappa$-two-convexity in the Riemannian setting. In the special case $\kappa = 0$, we will write 
\[G = \Big ( \sum_{i < j}\frac{1}{\lambda_i+\lambda_j} \Big )^{-1}.\] 
Finally, we denote by $\mu$ the reciprocal of the inscribed radius; that is, 
\[\mu (x,t)=\sup_{y\in M, 0<d(F(x,t),F(y,y))\leq \frac{1}{2} \, \text{\rm inj}(X)} \Big ( -\frac{2 \langle \exp_{F(x,t)}^{-1}(F(y,t)),\nu (x,t) \rangle}{d(F(x,t),F(y,t))^2} \Big ).\]
Our main result gives a sharp upper bound for $\mu$ in terms of the velocity $G_\kappa$. 

\begin{theorem}
\label{main.theorem}
Let $F:M^n\times [0,T)\to X^{n+1}$ be a one-parameter family of closed, embedded, $\kappa$-two-convex hypersurfaces which move with velocity $G_\kappa$. Then for any $\delta>0$ there exists a constant $C_1$, depending only on $\delta, T, M_0, X$, such that 
\[\mu \leq \Big ( \frac{(n-1)(n+2)}{4}+\delta \Big ) G_\kappa+C_1. \]
\end{theorem}

Note that $\mu = \frac{(n-1)^2(n+2)}{4} \, G$ on a cylinder, so the constant in Theorem \ref{main.theorem} is sharp. Thus, Theorem \ref{main.theorem} is of a similar nature as the cylindrical estimate in \cite{Brendle-Huisken}, but provides additional information about the formation of necks. Like the results in \cite{Brendle-Huisken}, our estimate extends with straightforward modifications to a larger class of fully nonlinear flows; see \cite[Remark 1.3]{Brendle-Huisken}. 

We briefly sketch the proof of Theorem \ref{main.theorem}. In a first step, we establish a sharp upper bound for the largest eigenvalue of the second fundamental form. More precisely, we show that $\lambda_n\leq (\frac{(n-1)(n+2)}{4}+\delta)G_\kappa+C_0$. This inequality is a variant of the cylindrical estimate in \cite{Brendle-Huisken}, and can be proven using integral estimates and Stampacchia as in \cite[Section 3]{Brendle-Huisken}. Alternatively, we can prove this estimate by a contradiction argument, using the pointwise curvature derivative estimate from \cite{Brendle-Huisken}. We will follow the latter approach here. Having established the upper bound for $\lambda_n$, we then use integral estimates and Stampacchia iteration to obtain a sharp upper bound for $\mu$. This step uses the evolution equation for the function $\mu$ (see \cite{Andrews-Langford-McCoy} or \cite[Section 4]{Brendle-Huisken}). Another crucial ingredient is an estimate for $\Delta \mu$ from \cite{Brendle1},\cite{Brendle2}. This inequality is independent of any evolution equation, and only makes use of the almost convexity property.

\section{Overview of results from \cite{Brendle-Huisken}}

In \cite{Brendle-Huisken} many important properties, including convexity estimates, cylindrical estimate, and curvature derivative estimate are established. For the reader's convenience, we recall some theorems in \cite{Brendle-Huisken} which we will use in the following sections.

\begin{lemma}
Consider a family of hypersurfaces moving with velocity $G$. If $X=\mathbb{R}^{n+1}$, under this flow we have basic evolution equations
\begin{align*}
\frac{\partial g_{ij}}{\partial t}&=-2Gh_{ij}\\
\frac{\partial h^i_{\ j}}{\partial t}&=\nabla^i\nabla_j G+Gh^i_{\ p}h^p_{\ j}\\
&=\frac{\partial G}{\partial h_{kl}} (\nabla_k\nabla_l h^i_{\ j}+h^p_{\ k}h_{pl}h^i_{\ j})+\frac{\partial^2 G}{\partial h_{kl}\partial h_{pq}}\nabla^i h_{kl}\nabla_j h_{pq}\\
\frac{\partial G}{\partial t}&=\frac{\partial G}{\partial h_{kl}} (\nabla_k\nabla_l G+h^p_{\ k}h_{pl}G).
\end{align*}
For a general ambient Riemannian manifold $X$ we have
\[\frac{\partial G_\kappa}{\partial t}=\frac{\partial G_\kappa}{\partial h_{kl}} (\nabla_k\nabla_l G_\kappa+h^p_{\ k}h_{pl}G_\kappa +\bar{R}_{k0l0} G_\kappa).\]
\end{lemma}
It is shown in \cite{Brendle-Huisken} that $G_\kappa$, $H$ and $|h|$ are comparable along the flow.

\begin{theorem}\label{cylindrical} \cite[Theorem 3.1]{Brendle-Huisken}
For any $\delta>0$ there exists a constant C depending on $\delta,T$ and the initial hypersurface $M_0$ such that
\begin{equation}
H \leq (\frac{(n-1)^2(n+2)}{4}+\delta)G_\kappa+C.
\end{equation}
\end{theorem}

\begin{corollary}\label{convexity} \cite[Corollary 3.7]{Brendle-Huisken}
For any $\delta>0$ there exists a constant $K_1$ depending on $\delta, T$ and the initial hypersurface $M_0$ such that
\begin{equation}
\lambda_1\geq -\delta G_\kappa-K_1.
\end{equation}

\end{corollary}
\begin{theorem}\label{gradient} \cite[Theorem 6.2]{Brendle-Huisken}
There exist constants $C_\#$ and $G_\#$ such that 
\begin{equation}
|\nabla h| \leq C_\# \, G_\kappa^2, \qquad |\nabla^2 h| \leq C_\# \, G_\kappa^3
\end{equation} 
whenever $G_\kappa \geq G_\#$.
\end{theorem}

\section{A sharp bound for the largest curvature eigenvalue}

In this section, we prove a sharp estimate for the largest eigenvalue of the second fundamental form. We begin with an auxiliary result concerning flows in Euclidean space moving with velocity $G$.

\begin{proposition}\label{constant}
Let $F:M^n\times (-\theta,0]\to\mathbb{R}^{n+1}$, be a one-parameter family of embedded, weakly convex hypersurfaces which move with velocity $G$. Suppose there is a constant $\beta>0$ such that $h_{ij}\leq \beta Gg_{ij}$ in $M^n\times (-\theta,0]$ and $\lambda_n (x_0,t_0)=\beta$ at some point $(x_0,t_0)\in M^n\times (-\theta,0]$, then $\beta\leq \frac{(n-1)(n+2)}{4}$.
\end{proposition}

\begin{proof}
Denote by $u_{ij}$ the two-tensor $\beta Gg_{ij}-h_{ij}$. The evolution equation of $u_{ij}$ is
\[\frac{\partial}{\partial t}u^i_{\ j} =\frac{\partial G}{\partial h_{kl}}\nabla_k\nabla_l u^i_{\ j}+\frac{\partial G}{\partial h_{kl}}h_{\ k}^p h_{pl}u^i_{\ j}-\frac{\partial^2 G}{\partial h_{kl}\partial h_{pq}}\nabla^{i}h_{kl}\nabla_j h_{pq},\]
and the last term on the right hand side is nonnegative since $G$ is concave. The strong maximum principle applied to $u_{ij}$ implies that the smallest eigenvalue of $u_{ij}$ is equal to $0$ at each point. (A detailed proof is given in the appendix.) Thus, $\beta G-\lambda_n = 0$ at each point. 

\begin{claim}
$\lambda_n$ is a spatial constant.
\end{claim}

To prove the claim, let us fix an arbitrary point $(x_0,t_0) \in M \times (\theta,0]$. We want to show that $\nabla \lambda_n=0$ at this point. Denote by $E$ the eigenspace of $h_{ij}$  with eigenvalue $\lambda_n = \beta G$. We break the discussion into two cases, depending on the dimension of $E$.\\

\textit{Case 1:} Suppose first that $\dim E=1$. In this case, we can choose a normal coordinate at $(x_0,t_0)$ such that at $(x_0,t_0)$ we have
\begin{align*}
\lambda_n&=h^n_{\ n}\\
\frac{\partial}{\partial t}\lambda_n &=\frac{\partial}{\partial t}h^n_{\ n}\\
\nabla_i \lambda_n &=\nabla_j h^n_{\ n}\\
\nabla_i \nabla_j \lambda_n &=\nabla_i \nabla_j h^n_{\ n}+\sum_{k \neq n} \frac{2}{\lambda_n-\lambda_k} \nabla_i h_{kn} \nabla_j h_{kn}.
\end{align*}
Then at this point we have
\begin{align*}
0&\geq \Big ( \frac{\partial}{\partial t}-\frac{\partial G}{\partial h_{kl}}\nabla_k\nabla_l-\frac{\partial G}{\partial h_{kl}} h^p_{\ k}h_{pk} \Big ) (\beta G-\lambda_n)\\
 &=-\frac{\partial^2 G}{\partial h_{kl}\partial h_{pq}}\nabla_n h_{kl}\nabla^n h_{pq}+  \frac{\partial G}{\partial h_{kl}}\sum_{j \neq n}\frac{2}{\lambda_n-\lambda_j} \nabla_k h_{nj} \nabla_l h_{nj} \geq 0.
\end{align*}
So we know $\nabla_n h_{kl}=\rho h_{kl}$ for some constant $\rho$ and $\nabla_k h_{nj}=0$ whenever $j\neq n$. Together with the Codazzi equation we get $\nabla_{n}h_{ij}=0$ and $\nabla_{j}\lambda_n=\nabla_{n}h^n_{\ j}=0$.\\

\textit{Case 2:} Suppose next that $\dim E\geq 2$. Let $v,\tilde{v}$ be two orthonormal vectors in $E$. Extend $v$ to a unit vector field in spacetime such that $\nabla v = 0$ at $(x_0,t_0)$ and $\frac{\partial}{\partial t} v^i = G h_j^i v^j$ at $(x_0,t_0)$. Then we have
\begin{align*}
0 &\geq \Big ( \frac{\partial}{\partial t}-\frac{\partial G}{\partial h_{kl}}\nabla_k\nabla_l-\frac{\partial G}{\partial h_{kl}} h^p_{\ k}h_{pk} \Big ) (\beta G - h_{ij} v^i v^j) \\
&=-\frac{\partial^2 G}{\partial h_{kl}\partial h_{pq}}\nabla_i h_{kl}\nabla_j h_{pq}v^iv^j \geq 0,
\end{align*}
Consequently, $v^k\nabla_k h_{ij}=\rho h_{ij}$ and similarly $\tilde{v}^k\nabla_{k} h_{ij}=\tilde{\rho} h_{ij}$. From the Coddazzi equation
\[\rho \lambda_n = \tilde{v}^i\tilde{v}^j v^k \nabla_{k} h_{ij}=\tilde{v}^i\tilde{v}^j v^k \nabla_i h_{kj}= \tilde{\rho} \tilde{v}^j v^k h_{kj}=0. \]
Therefore, $\rho=0$ and $v^k\nabla_k h_{ij}=0$ for all $v\in E$. In particular, $v^i v^j \nabla_k h_{ij}=0$ by the Codazzi equations. Since this holds for every unit vector $v \in E$, the claim follows. \\

We now continue with the proof of Proposition \ref{constant}. Since $\lambda_n$ is a spatial constant and $\beta G - \lambda_n = 0$, it follows that $G$ is also a spatial constant. From
\begin{align*}
\frac{\partial}{\partial t} h^i_{\ j} =\nabla^i\nabla_j G+Gh^i_{\ p}h^p_{\ j}
\end{align*}
the diagonalization of $h_{ij}$ is preserved under the flow and $\lambda_j$ satisfies an ODE
\[ \frac{\partial}{\partial t}\lambda_j=G\lambda_j^2. \]
Thus 
\begin{align*}
0&=\frac{\partial}{\partial t}(\beta G-\lambda_n)=\beta \frac{\partial G}{\partial \lambda_j}\lambda_j^2 G-G\lambda_n^2 \\
&=\lambda_n \Big ( \frac{\partial G}{\partial \lambda_j}\lambda_j (\lambda_j-\lambda_n) \Big ) \leq 0,
\end{align*} 
where in the last step we have used that the flow is weakly convex. Consequently, for each $j$, we either have $\lambda_j=0$ or $\lambda_n=n$.  In view of two-convexity, the only two possibilities are 
\[ \lambda_1=0, \, \lambda_2=\dots =\lambda_n, \, \beta=\frac{(n-1)(n+2)}{4}\]
or
\[\lambda_1=\lambda_2=\dots =\lambda_n, \, \beta=\frac{n(n-1)}{4}.\]
In each case, $\beta\leq \frac{(n-1)(n+2)}{4}$. 
\end{proof}

\begin{proposition}\label{cylinder}
Let $F:M^n\times [0,T)\to X^{n+1}$ be a one-parameter family of embedded, $\kappa$-two-convex hypersurfaces in $X^{n+1}$ moving with velocity $G_\kappa$. Then for any $\delta>0$ there is a constant $C_0$, depending only on $\delta,T,M_0,X$, such that $\lambda_n\leq (\frac{(n-1)(n+2)}{4}+\delta)G_\kappa+C_0$.
\end{proposition}
\begin{proof}
Define 
\[\beta=\inf_{K>0}\sup_{\{G_\kappa>K\}} \frac{\lambda_n}{G_\kappa}.\]
We claim that $\beta\leq \frac{(n-1)(n+2)}{4}$. To prove this, we use a blow-up argument. Take $(x_k,t_k)$ be a sequence of spacetime points such that $G_\kappa(x_k,t_k)\to \infty$ and $\lambda_n(x_k,t_k)/G_\kappa(x_k,t_k) \to \beta$. By Theorem \ref{gradient} the second fundamental form and its derivatives are uniformly bounded in a parabolic neighborhood $\mathcal{P}(x_k,t_k,\theta G_\kappa^{-1},\theta^2 G_\kappa(x_k,t_k)^{-2})$. By doing a parabolic dilation with center $(x_k,t_k)$ and scale $G_\kappa(x_k,t_k)$, we obtain a smooth limit flow in Euclidean space which moves with velocity $G$. This limit flow is defined on some time interval $(-\theta,0]$, where $\theta$ depends on the constant $C_\#$ in Theorem \ref{gradient}. Along the limit flow $\beta G g_{ij}-h_{ij}\geq 0$ and $\beta G-\lambda_n=0$ at $(0,0)$. By Proposition \ref{constant}, $\beta\leq \frac{(n-1)(n+2)}{4}$. Having established that $\beta \leq \frac{(n-1)(n+2)}{4}$, the assertion follows easily.
\end{proof}

\section{$L^p$ estimates for $\mu$}

In \cite[Proposition 4.1]{Brendle-Huisken} it is proved that
\begin{equation}
\frac{\partial \mu}{\partial t}\leq \frac{\partial G_\kappa}{\partial h_{kl}} (\nabla_k\nabla_l \mu+h^p_{\ k}h_{pl}\mu)-\sum_{i=1}^n\frac{1}{\mu-\lambda_i}\frac{\partial G_\kappa}{\partial \lambda_i}(\nabla_i \mu)^2+C\mu+C\sum_{i=1}^n\frac{1}{\mu-\lambda_i}
\end{equation}
in the viscosity sense in the domain $\{\mu>\lambda_n \}\cap \{\mu \geq 8 \, \text{\rm inj}(X)^{-1}\}$. For any $1>\sigma>0$ define the function
\begin{align*}
 f_\sigma&=G_\kappa^{\sigma-1} \Big ( \mu - \Big ( \frac{(n-1)(n+2)}{4}+\delta \Big ) G_\kappa \Big )-K_0\\
 f_{\sigma,+}&=\max \{f_\sigma ,0\},
\end{align*}
for some constant $K_0$ to be determined. By choosing $K_0$ large enough, depending $M_0$ and $C_0$ in Proposition \ref{cylinder}, one can arrange that $\mu-\lambda_n\geq \frac{\delta}{2} G_\kappa$ on the set $\{f_\sigma \geq 0\}$. In addition, by choosing $K_0$ large enough, depending on $\min_{M_0} G_\kappa$ and the injective radius of $X$, we can arrange that $\mu\geq 8 \, \text{\rm inj}(X)^{-1}$ on the set $\{f_\sigma \geq 0\}$. Finally, if we choose $K_0$ large enough, then $G_\kappa \geq G_\#$ on the set $\{f_\sigma \geq 0\}$. In the following, $C$ denotes a constant depending on $\delta, M_0, C_0$ in Proposition \ref{cylinder} and $C_\#$ in Theorem \ref{gradient} but not on $\sigma,p$. With the lower bound of $G_\kappa$, one has 
\begin{equation}
\frac{\partial \mu}{\partial t}\leq \frac{\partial G_\kappa}{\partial h_{kl}}(\nabla_k\nabla_l \mu+h^p_{\ k}h_{pl}\mu)-\sum_{i=1}^n\frac{1}{\mu-\lambda_i}\frac{\partial G_\kappa}{\partial \lambda_i}(\nabla_i \mu)^2+CG_\kappa
\end{equation} 
on the set $\{f_\sigma \geq 0\}$.

\begin{proposition}\label{Lp.bound}
There exists a small positive constant $c_0$, depending only on $\delta,T,M_0,X$, with the following significance. If $p \geq \frac{1}{c_0}$ and $\sigma \leq c_0 \, p^{-\frac{1}{2}}$, then we have 
\[\int_{M_t} f_{\sigma,+}^p \leq C,\] 
where $C$ is a positive constant that depends only on $p,\sigma,\delta,M_0, X$.
\end{proposition} 

\begin{proof}
As in \cite{Brendle1}, $f_\sigma$ satisfies the evolution equation
\begin{align*}
\frac{\partial}{\partial t} f_\sigma &\leq \frac{\partial G_\kappa}{\partial h_{kl}}\nabla_k\nabla_l f_\sigma-G_\kappa^{\sigma-1}\sum_{i=1}^n\frac{1}{\mu-\lambda_i}\frac{\partial G_\kappa}{\partial \lambda_i}(\nabla_i \mu)^2 + CG_\kappa^\sigma \\ 
&+2(1-\sigma)G_\kappa^{-1}\frac{\partial G_\kappa}{\partial h_{kl}}\nabla_k G_\kappa\nabla_l f_\sigma+\sigma\frac{\partial G_\kappa}{\partial h_{kl}}h^p_{\ k}h_{pl}(f_\sigma+K_0) 
\end{align*} 
on the set $\{f_\sigma \geq 0\}$. We also note that $f_\sigma \leq C \, G_\kappa^\sigma$. Together with the curvature derivative estimate in Theorem \ref{gradient}, we obtain
\begin{align*}
\frac{d}{dt} \bigg (\int_{M_t}f^p_{\sigma,+} \bigg )&\leq -\frac{p^2}{C}\int_{M_t}f^{p-2}_{\sigma,+}|\nabla f_{\sigma,+} |^2-\frac{p}{C}\int_{M_t} G_\kappa^{-2} f_{\sigma,+}^p |\nabla\mu|^2 \\ 
&+ Cp \int_{M_t} G_\kappa^\sigma f_{\sigma,+}^{p-1} + Cp\int_{M_t} G_\kappa f_{\sigma,+}^{p-1} |\nabla f_{\sigma,+}| \\ 
&+C\sigma p\int_{M_t} G_\kappa^2f_{\sigma,+}^{p-1}(f_{\sigma,+}+K_0)\\
&\leq -\frac{p^2}{C}\int_{M_t}f^{p-2}_{\sigma,+}|\nabla f_{\sigma,+} |^2-\frac{p}{C}\int_{M_t}G_\kappa^{-2} f_{\sigma,+}^p |\nabla\mu|^2 \\ 
&+Cp \int_{M_t} G_\kappa^\sigma f_{\sigma,+}^{p-1}+C(\sigma p+1) \int_{M_t} G_\kappa^2f_{\sigma,+}^{p-1}(f_{\sigma,+}+K_0).
\end{align*}
Since $G_\kappa$ is uniformly bounded from below, we have the pointwise estimate 
\[p G_\kappa^\sigma f_{\sigma,+}^{p-1} - G_\kappa^2 f_{\sigma,+}^p \leq (Cp)^p \, G_\kappa^{2-(2-\sigma) \, p} \leq (Cp)^p.\] 
This gives 
\begin{align}\label{ev}
\begin{split}
\frac{d}{dt} \bigg (\int_{M_t}f^p_{\sigma,+} \bigg )
&\leq -\frac{p^2}{C}\int_{M_t}f^{p-2}_{\sigma,+}|\nabla f_{\sigma,+} |^2-\frac{p}{C}\int_{M_t} G_\kappa^{-2} f_{\sigma,+}^p |\nabla\mu|^2 \\ 
&+C(\sigma p+1) \int_{M_t} G_\kappa^2f_{\sigma,+}^{p-1}(f_{\sigma,+}+K_0) + (Cp)^p \, |M_t|.
\end{split}
\end{align} 
At this point, we apply Corollary 3.3 in \cite{Brendle2}. This inequality is independent of any evolution equation, and only uses the convexity estimate in Corollary \ref{convexity}. This implies
\begin{align*}
0&\leq -\int_{M_t} \langle \nabla\eta,\nabla\mu \rangle+\frac{1}{2}\int_{M_t}\eta (|h|^2\mu-H\mu^2+n^3(n\varepsilon\mu+K_1(\varepsilon)\mu^2))\\
&+\int_{M_t}\eta\sum_{i=1}^n \frac{1}{\mu-\lambda_i}( |\nabla_i\mu |+C) |\nabla_i H |\\
&+\int_{M_t}\eta (H+n^3(n\varepsilon\mu+K_1(\varepsilon))) \sum_{i=1}^n\frac{1}{(\mu-\lambda_i)^2}((\nabla_i\mu)^2+C) \\ 
&+ C \int_{M_t} \eta \mu + C \int_{M_t} \eta \sum_{i=1}^n \frac{1}{\mu-\lambda_i}
\end{align*}
for any nonnegative test function $\eta$ with support contained in $\{\mu>\lambda_n\}\cap \{ \mu\geq 8 \, \text{\rm inj}(X)^{-1} \}$. Taking $\eta=\frac{f_{\sigma,+}^p}{G_\kappa}$ as the test function yields
\begin{align}\label{inhomo}
\begin{split}
&\frac{1}{2}\int_{M_t} (H\mu^2-|h|^2\mu-n^3 (n\varepsilon\mu+K_1(\varepsilon))\mu^2) \frac{f_{\sigma,+}^p}{G_\kappa}\\
&\leq -\int_{M_t} \langle \nabla  (\frac{f_{\sigma,+}^p}{G_\kappa}),\nabla\mu \rangle \\
&+\int_{M_t}\sum_{i=1}^n\frac{1} {\mu-\lambda_i} (|\nabla_i\mu|+C)|\nabla_i H|\frac{f_{\sigma,+}^p}{G_\kappa}\\
&+\frac{1}{2}\int_{M_t} (H+n^3 (n\varepsilon\mu +K_1(\varepsilon)))\sum_{i=1}^n\frac{1}{(\mu-\lambda_i)^2} \, ((\nabla_i\mu)^2+C) \, \frac{f_{\sigma,+}^p}{G_\kappa} \\ 
&+ C \int_{M_t} f_{\sigma,+}^p+C \int_{M_t}\frac{f_{\sigma,+}^p}{G_\kappa}\sum_{i=1}^n\frac{1}{\mu-\lambda_i}.
\end{split}
\end{align}
Using the fact that $\mu\geq \lambda_n+\frac{\delta}{2}G_\kappa$ on the set $\{f_\sigma \geq 0\}$ and the curvature derivative estimate, the right hand side of (\ref{inhomo}) can be estimated by
\begin{align*}
RHS&\leq Cp\int_{M_t} G_\kappa^{-1}f_{\sigma,+}^{p-1}|\nabla f_{\sigma,+}||\nabla \mu|+C\int_{M_t} f_{\sigma,+}^p|\nabla\mu|\\
&+C\int_M G_\kappa^{-2}f_{\sigma,+}^p|\nabla\mu|^2+C\int_{M_t} f_{\sigma,+}^p.
\end{align*}
Using Corollary \ref{convexity} together with the fact that $\mu-\lambda_n \geq \frac{\delta}{2} G_\kappa$ on the set $\{f_\sigma \geq 0\}$, we obtain  
\begin{align*}
|h|^2 &\leq \sum_{i=2}^n\lambda_i\lambda_n+\lambda_1^2\\ 
&=H\lambda_n+\lambda_1(\lambda_1-\lambda_n)\\
&\leq H\mu-\frac{\delta}{4}HG_\kappa+C.
\end{align*}
Thus, by choosing $\varepsilon$ small enough, we obtain
\[H\mu^2-|h|^2\mu-n^3(n\varepsilon\mu+K_1(\varepsilon))\mu^2 \geq \frac{\delta}{8}HG_\kappa\mu-CG_\kappa.\]
Hence, the left hand side of  (\ref{inhomo}) can be estimated by  
\begin{align*}
LHS\geq \int_M\frac{\delta}{8}H\mu f_{\sigma,+}^p-C \int_M f_{\sigma,+}^p.
\end{align*}
Putting these facts together gives 
\begin{align*} 
\int_{M_t} G_\kappa^2 f_{\sigma,+}^p 
&\leq Cp\int_{M_t} G_\kappa^{-1}f_{\sigma,+}^{p-1}|\nabla f_{\sigma,+}||\nabla \mu|+C\int_{M_t} f_{\sigma,+}^p|\nabla\mu|\\
&+C\int_M G_\kappa^{-2}f_{\sigma,+}^p|\nabla\mu|^2+C\int_{M_t} f_{\sigma,+}^p.
\end{align*} 
Using Young's inequality, we can absorb the term $\int_{M_t} f_{\sigma,+}^p|\nabla\mu|$ into the left hand side. This gives
\begin{align*} 
\int_{M_t} G_\kappa^2 f_{\sigma,+}^p 
&\leq Cp\int_{M_t} G_\kappa^{-1}f_{\sigma,+}^{p-1}|\nabla f_{\sigma,+}||\nabla \mu|+C\int_M G_\kappa^{-2}f_{\sigma,+}^p|\nabla\mu|^2 \\ 
&+C\int_{M_t} f_{\sigma,+}^p.
\end{align*} 
Combining this with the pointwise inequality $f_{\sigma,+}^{p-1}(f_{\sigma,+}^p+K_0)\leq 2f_{\sigma,+}^p+K_0^p$, we deduce that
\begin{align}\label{ax}
\begin{split}
&\int_{M_t}G_\kappa^2 f_{\sigma,+}^{p-1}(f_{\sigma,+}+K_0) \\ 
&\leq Cp\int_{M_t} G_\kappa^{-1}f_{\sigma,+}^{p-1}|\nabla f_{\sigma,+}||\nabla \mu|+C\int_M G_\kappa^{-2}f_{\sigma,+}^p|\nabla\mu|^2 \\ 
&+C\int_{M_t} f_{\sigma,+}^p + K_0^p \int_{M_t} G_\kappa^2. 
\end{split}
\end{align}
From (\ref{ev}), (\ref{ax}), we conclude that if $p\geq \frac{1}{c_0}$ and $\sigma\leq c_0 p^{-\frac{1}{2}}$ then
\begin{align*}
&\frac{d}{dt} \bigg ( \int_{M_t}f_{\sigma,+}^p \bigg ) \\
&\leq -\frac{p^2}{C}\int_{M_t}f_{\sigma,+}^{p-2}|\nabla f_{\sigma,+}|^2-\frac{p}{C}\int_{M_t} G_\kappa^{-2} f_{\sigma,+}^p |\nabla\mu|^2\\
&+C(\sigma p^2+p) \int_{M_t}G_\kappa^{-1}f_{\sigma,+}^{p-1}|\nabla f_{\sigma,+}||\nabla \mu|+C(\sigma p+1)\int_{M_t}G_\kappa^{-2}f_{\sigma,+}^p|\nabla \mu |^2\\
&+ C(\sigma p+1)\int_{M_t}f_{\sigma,+}^p+(\sigma p+1)K_0^p \int_{M_t}G_\kappa^2 + (Cp)^p \, |M_t| \\ 
&\leq C(\sigma p+1)\int_{M_t}f_{\sigma,+}^p+(\sigma p+1)K_0^p \int_{M_t}G_\kappa^2 + (Cp)^p \, |M_t|. 
\end{align*}
Since $\int_0^T \int_{M_t} G_\kappa^2 \leq C \int_0^T \int_{M_t} H^2 \leq C \, |M_0|$, the assertion follows.
\end{proof}

The following result is a direct consequence of the proof of Proposition \ref{Lp.bound}:

\begin{proposition}\label{fin}
Let $f_{\sigma,k}=G_\kappa^{\sigma-1} \Big ( \mu-(\frac{(n-1)(n+2)}{4}+\delta)G_\kappa \Big )-k$ and $f_{\sigma,k,+}=\max \{ f_{\sigma,k},0\}$. Then for $k\geq K_0$, $p\geq \frac{1}{c_0}$  and $ \sigma\leq c_0 p^{-\frac{1}{2}}$ we have
\begin{align*} 
\frac{d}{dt} \bigg ( \int_{M_t}f_{\sigma,k,+}^p \bigg ) 
&\leq -\frac{p^2}{C}\int_{M_t}f_{\sigma,k,+}^{p-2}|\nabla f_{\sigma,k}|^2 \\ 
&+C(\sigma p+1) \int_{M_t}G_\kappa^2f_{\sigma,k,+}^{p-1}(f_{\sigma,k}+k) + (Cp)^p \, |M_t\cap \{f_{\sigma,k}\geq 0\}|. 
\end{align*}
\end{proposition}

\section{Stampacchia iteration}
From now on we fix $p$ and $\sigma$ such that $p \geq \frac{1}{c_0}$ and $\sigma+2p^{-1} \leq c_0 \, (2np)^{-\frac{1}{2}}$. In the following, $C$ will denote a constant that may depend on $p$ and $\sigma$. By Proposition \ref{Lp.bound}, we have 
\[ \int_0^T\int_{M_t} (f_{\sigma,0,+}^{2np}+f_{\sigma+2p^{-1},0,+}^{2np}) dt\leq C. \]
Let
\[ A(k)=\int_0^T\int_{M_t}1_{\{f_{\sigma,k}\geq 0 \}}dt. \] 
For any $k$ sufficiently large, Proposition \ref{fin} implies
\begin{align*}
\sup_{t\in [0,T)} \int_{M_t}f_{\sigma,k,+}^p &\leq C\int_0^T\int_{M_t}G_\kappa^2f_{\sigma,k,+}^{p-1}(f_{\sigma,k,+}+k) dt + CA(k) \\
\int_0^T\int_{M_t}f_{\sigma,k,+}^{p-2}|\nabla f_{\sigma,k}|^2 dt &\leq C\int_0^T\int_{M_t}G_\kappa^2f_{\sigma,k,+}^{p-1}(f_{\sigma,k,+}+k) dt + CA(k).
\end{align*}
Applying the  Michael-Simon Sobolev inequality (cf. \cite{Michael-Simon}) to $f_{\sigma,k,+}^p$, we obtain 
\[\bigg ( \int_{M_t}f_{\sigma,k,+}^{\frac{pn}{n-1}} \bigg )^{\frac{n-1}{n}} \leq C \int_{M_t} (f_{\sigma,k,+}^{p-1}|\nabla f_{\sigma,k}|+G_\kappa f_{\sigma,k,+}^p),\] 
hence 
\begin{align*}
\int_0^T \bigg ( \int_{M_t}f_{\sigma,k,+}^{\frac{pn}{n-1}} \bigg )^{\frac{n-1}{n}}dt &\leq C\int_0^T \int_{M_t} (f_{\sigma,k,+}^{p-1}|\nabla f_{\sigma,k}|+G_\kappa f_{\sigma,k,+}^p) dt\\
&\leq C \int_0^T \int_{M_t} (f_{\sigma,k,+}^{p-2}|\nabla f_{\sigma,k}|^2+(G_\kappa^2+1) f_{\sigma,k,+}^p) dt\\
&\leq C  \int_0^T \int_{M_t} (G_\kappa^2+1) f_{\sigma,k,+}^{p-1} (f_{\sigma,k}+k) dt + CA(k).
\end{align*}
H\"older's inequality gives 
\[\int_{M_t} f_{\sigma,k,+}^{\frac{p(n+1)}{n}} \leq \bigg ( \int_{M_t}f_{\sigma,k,+}^{p} \bigg )^{\frac{1}{n}} \bigg ( \int_{M_t}f_{\sigma,k,+}^{\frac{pn}{n-1}} \bigg )^{\frac{n-1}{n}},\] 
hence 
\begin{align*} 
\int_0^T\int_{M_t} f_{\sigma,k,+}^{\frac{p(n+1)}{n}} dt 
&\leq \bigg ( \sup_{t\in [0,T)} \int_{M_t}f_{\sigma,k,+}^{p} \bigg )^{\frac{1}{n}} \int_0^T \bigg ( \int_{M_t}f_{\sigma,k,+}^{\frac{pn}{n-1}} \bigg )^{\frac{n-1}{n}}dt\\
&\leq \bigg ( C \int_0^T \int_{M_t} (G_\kappa^2+1) f_{\sigma,k,+}^{p-1} (f_{\sigma,k}+k) dt + C A(k) \bigg )^{\frac{n+1}{n}}.
\end{align*} 
Therefore, 
\begin{align*}
&\bigg ( \int_0^T\int_{M_t} f_{\sigma,k,+}^{\frac{p(n+1)}{n}} dt \bigg )^{\frac{n}{n+1}} \\ 
&\leq C \int_0^T \int_{M_t} (G_\kappa^2+1) f_{\sigma,0,+}^{p} 1_{\{f_{\sigma,k} \geq 0\}} dt + C A(k) \\
&\leq C \bigg ( \int_0^T \int_{M_t} (G_\kappa^{4n}+1) f_{\sigma,0,+}^{2np} dt \bigg )^{\frac{1}{2n}} \bigg ( \int_0^T \int_{M_t} 1_{\{f_{\sigma,k} \geq 0\}} \bigg )^{\frac{2n-1}{2n}}+ C A(k)\\
&\leq C \bigg ( \int_0^T \int_{M_t} (f_{\sigma,0,+}^{2np}+f_{\sigma+2p^{-1},0,+}^{2np}) dt \bigg )^{\frac{1}{2n}} A(k)^{\frac{2n-1}{2n}}+ C A(k)\\
&\leq C A(k)^{\frac{2n-1}{2n}}.
\end{align*}
Consequently,
\[ (\tilde{k}-k)^{\frac{p(n+1)}{n}}A(\tilde{k})\leq C A(k)^{\frac{2n^2+2n-1}{2n^2}}.\] 
Iterating this inequality gives $A(k)=0$ for $k$ large enough (cf. \cite[Lemma 4.1]{Stampacchia}).

\appendix 

\section{Appendix}

In this section we show that strong maximum principle holds for $u_{ij}=\beta G g_{ij}-h_{ij}$.
\begin{proposition}
Let $\Omega$  be a bounded domain with smooth boundary. Assume $u_{ij}$ is non-negative in $\Omega\times [0,T]$ and $\lambda_1 (u)=0$ at a point $(x_0,T)$, then $\lambda_1 (u)=0$ in $\Omega\times [0,T]$
\end{proposition}
\begin{proof}
It suffices to prove $\lambda_1(u)=0$ at time $t=0$. We argue by contradiction. Suppose that $\lambda_1(u(y,0))>0$ at some point $y \in \Omega$. We can find a non-negative smooth function $f_0$ such that $f_0(x) \leq \lambda_1(u(x,0))$ for each point $x \in \Omega$, $f(y) \geq \frac{1}{2} \, \lambda_1(u(y,0))$, and $f(x) = 0$ for $x \in \partial \Omega$. Let us choose a large positive constant $A$ such that $A+\frac{\partial G}{\partial h_{kl}}h_{\ k}^ph_{pl}\geq 0$. Let $f$ be the solution of the equation 
\[\frac{\partial f}{\partial t}=\frac{\partial G}{\partial h_{kl}}\nabla_k\nabla_l f-Af\] 
with initial condition $f(x,0)=f_0(x)$ for $x \in \Omega$ and boundary condition $f(x,t) = 0$ for $x \in \partial \Omega$. By the strong maximum principle for scalar functions, $f>0$ in $\Omega\times (0,T]$. Since $u_{ij}$ satisfies the inequality 
\[\frac{\partial}{\partial t}u^i_{\ j} \geq \frac{\partial G}{\partial h_{kl}}\nabla_k\nabla_l u^i_{\ j}+\frac{\partial G}{\partial h_{kl}}h_{\ k}^p h_{pl}u^i_{\ j},\] 
the tensor $\tilde{u}_{ij} = u_{ij} - f g_{ij}$ satisfies  
\begin{align*}
\frac{\partial}{\partial t} \tilde{u}^i_{\ j}
&\geq \frac{\partial G}{\partial h_{kl}}\nabla_k\nabla_l \tilde{u}^i_{\ j}+\frac{\partial G}{\partial h_{kl}}h_{\ k}^p h_{pl}\tilde{u}^i_{\ j} \\
&+ \Big ( \frac{\partial G}{\partial h_{kl}}h_{\ k}^ph_{pl}+A \Big ) f \delta^i_{\ j} \\
&\geq \frac{\partial G}{\partial h_{kl}}\nabla_k\nabla_l \tilde{u}^i_{\ j}+\frac{\partial G}{\partial h_{kl}}h_{\ k}^p h_{pl}\tilde{u}^i_{\ j}.
\end{align*}
By the weak maximum principle for tensors (cf. \cite[Theorem 9.1]{Hamilton1}), $\tilde{u}\geq 0$ in $\Omega\times [0,T]$. Thus $\lambda_1(u(x_0,T)) \geq f(x_0,T)>0$ which contradicts the assumption.
\end{proof}


\begin{thebibliography}{99}
\bibitem{Andrews} 
B.~Andrews, \textit{Non-collapsing in mean-convex mean curvature flow,} Geom. Topol. 16, 1413--1418 (2012)

\bibitem{Andrews-Langford-McCoy}
B.~Andrews, M.~Langford, and J.~McCoy, \textit{Non-collapsing in fully non-linear curvature flows,} Ann. Inst. H. Poincar\'e Anal. Non Lin\'eaire 30, 23--32 (2013)

\bibitem{Brendle-survey}
S.~Brendle, \textit{Two-point functions and their applications in geometry,} Bull. Amer. Math. Soc. 51, 581--596 (2014)

\bibitem{Brendle1} 
S.~Brendle, \textit{A sharp bound for the inscribed radius under mean curvature flow,} Invent. Math. 202, 217--237 (2015)

\bibitem{Brendle2} 
S.~Brendle, \textit{An inscribed radius estimate for mean curvature flow in Riemannian manifolds,} Ann. Scuola Norm. Sup. Pisa 16, 1447--1472 (2016)

\bibitem{Brendle-Huisken} 
S.~Brendle and G.~Huisken, \textit{A fully nonlinear flow for two-convex hypersurfaces,} to appear in Invent. Math.

\bibitem{Hamilton1}
R.~Hamilton, \textit{Three-manifolds with positive Ricci curvature,} J. Diff. Geom. 17, 255--306 (1982)

\bibitem{Hamilton2}
R.~Hamilton, \textit{Isoperimetric estimates for the curve shrinking flow in the plane,} Modern Methods in Complex Analysis (Princeton 1992), 201--222, Ann. of Math. Stud. 137, Princeton University Press, Princeton NJ (1995)

\bibitem{Haslhofer-Kleiner}
R.~Haslhofer and B.~Kleiner, \textit{On Brendle's estimate for the inscribed radius under mean curvature flow,} Intern. Math. Res. Not. 15, 6558-6561 (2015)

\bibitem{Huisken}
G.~Huisken, \textit{A distance comparison principle for evolving curves,} Asian J. Math. 2, 127--133 (1998)

\bibitem{Michael-Simon}
J.H.~Michael and L.M.~Simon, \textit{Sobolev and mean value inequalities on generalized submanifolds of $\mathbb{R}^n$,} Comm. Pure Appl. Math. 26, 316--379 (1973)

\bibitem{Sheng-Wang}
W.~Sheng and X.J.~Wang, \textit{Singularity profile in the mean curvature flow,} Methods Appl. Anal. 16, 139--155 (2009) 

\bibitem{Stampacchia}
G.~Stampacchia, \textit{Equations elliptiques au second ordre \`a co\'efficients discontinues,} S\'eminaire de Math\'ematiques Sup\'erieures 16, Les Presses de l'Universit\'e de Montreal, Montreal, 1966

\bibitem{White1}
B.~White, \textit{The size of the singular set in mean curvature flow of mean convex sets,} J. Amer. Math. Soc. 13, 665--695 (2000)

\bibitem{White2}
B.~White, \textit{The nature of singularities in mean curvature flow of mean convex sets,} J. Amer. Math. Soc. 16, 123--138 (2003)
\end{thebibliography}
\end{document}